\documentclass[a4paper,english]{amsart}
\usepackage[T1]{fontenc}
\usepackage[latin9]{inputenc}
\usepackage{amsthm}
\usepackage{amstext}
\usepackage{setspace}
\usepackage{amssymb}
\usepackage{setspace}
\usepackage{float}
\usepackage{fullpage}
\usepackage{graphicx}
\usepackage{mathptmx}
\usepackage{psfrag}
\usepackage{epsfig}
\usepackage{color}
\usepackage{amscd}
\def\S{\mathbb{S}}
\makeatletter
\numberwithin{equation}{section} %% Comment out for sequentially-numbered
\numberwithin{figure}{section} %% Comment out for sequentially-numbered
\theoremstyle{plain}
\newtheorem{thm}{Theorem}[section]
  \theoremstyle{plain}
  \newtheorem{lem}[thm]{Lemma}

  \theoremstyle{remark}
  \newtheorem{rem}[thm]{Remark}
  \theoremstyle{plain}
  
\newtheorem{prop}[thm]{Proposition}
\usepackage{babel}
\usepackage[unicode=true,
 bookmarks=true,
 bookmarksnumbered=false,
 bookmarksopen=false,
 breaklinks=false,
 pdfborder={0 0 1},
 backref=false,
 colorlinks=false]
 {hyperref}
\hypersetup{%
 pdfauthor={Marc Kesseboehmer, Bernd O. Stratmann} }

\usepackage{amsthm}
\usepackage[all]{xy}
\usepackage{psfrag}
\usepackage{bbm}

\DeclareMathOperator{\var}{var}

\newcommand{\Z}{\mathbb{Z}}

\newcommand{\R}{\mathbb {R}}

\newcommand{\D}{\mathbb {D}}

\newcommand{\N}{\mathbb {N}}

\renewcommand{\hat}{\widehat}
\def\bewend{\hspace{\stretch{1}}$\square$}
\def\var{\varepsilon}

\makeatother

\usepackage{babel}

\begin{document}
\selectlanguage{english}

\title{On Hausdorff dimension and cusp excursions for Fuchsian groups}

%\author{Marc Kesseböhmer}

%%\address{Fachbereich 3 -- Mathematik und Informatik, Universität Bremen, Bibliothekstr.
%1, D--28359 Bremen, Germany}

%\email{mhk@math.uni-bremen.de}

\author{Sara Munday}

\address{Mathematical Institute, University of St. Andrews, North Haugh, St.
Andrews KY16 9SS, Scotland}

\email{sam20@st-and.ac.uk}

%\author{Bernd O. Stratmann}

%\address{Mathematical Institute, University of St. Andrews, North Haugh, St.
%Andrews KY16 9SS, Scotland}

%\email{bos@maths.st-and.ac.uk}

\maketitle
\begin{abstract}
Certain subsets of limit sets of geometrically finite Fuchsian groups with parabolic elements are considered. It is known that Jarn\'{\i}k limit sets determine a "weak multifractal spectrum" of the Patterson measure in this situation. This paper will describe a natural generalisation of these sets, called strict Jarn\'{\i}k limit sets, and show how these give rise to another weak multifractal spectrum. Number-theoretical interpretations of these results in terms of continued fractions will also be given.
\end{abstract}

\section{Introduction and statement of results}

%In this paper, we consider certain subsets of the limit set of a non-elementary, geometrically finite Fuchsian group $G$ with a parabolic element.
Let us begin by recalling some of the basic notions from hyperbolic geometry which will be used throughout the paper (for more details, the reader is referred to \cite{ab}). We will make use of both the Poincar\'{e} Disc model $(\D^2, d_h)$ and the upper half-plane model $(\mathbb{H}, d_{\mathbb{H}})$ of two-dimensional hyperbolic space and we will use $d$ to denote either of these metrics whenever the context does not depend on the particular model. The group of isometries of hyperbolic space is denoted by $Con(1)$. Recall that in the upper half-plane model, this group is isomorphic to the group $PSL_2(\R)$. It is well known that the elements of $Con(1)$ can be classified as parabolic, hyperbolic or elliptic according to their fixed points or, equivalently, geometric action. A parabolic element is one that has exactly one fixed point which lies on the boundary of hyperbolic space. Each parabolic element leaves invariant a family of \textit{horoballs}. In $\D^2$, these are Euclidean circles internally tangent to $\S^1$; the point of tangency is the fixed point of the parabolic map. In $\mathbb{H}$, these are either Euclidean circles tangent to $\R$ or, if the fixed point of the parabolic map is the point at infinity, Euclidean horizontal straight lines.

We shall consider discrete subgroups of the group $Con(1)$, that is, subgroups that act properly discontinuously on the interior of hyperbolic space. These groups are known as \textit{Fuchsian groups}. We are interested in subsets of the \textit{limit set} of a Fuchsian group. The limit set $L(G)$ is the set of accumulation points of the orbit of any point under the group $G$. It is necessarily contained in the boundary of hyperbolic space.
It is well known that the limit set of a Fuchsian group is either empty, consists of one or two points, or else contains uncountably many points. If $L(G)$ is uncountable, we say that $G$ is \textit{non-elementary}. A group $G$ is referred to as \textit{geometrically finite} if it has a fundamental domain with finitely many sides. In two-dimensions, a Fuchsian group is geometrically finite if and only if it is finitely generated. (Note that this is no longer true in higher dimensions.) Throughout this paper, we shall only be concerned with non-elementary, geometrically finite Fuchsian groups with one parabolic element. Let us note here that the restriction to one parabolic element is not  essential; this is addressed in Remark \ref{onepara} below.

So, suppose that $G$ is a non-elementary, geometrically finite Fuchsian groups with one parabolic element, say $\gamma$, and suppose that $p$ is the fixed point of this element $\gamma$.
We now describe a certain set of horoballs associated to the orbit of the parabolic fixed point $p$ under the group $G$, called a \textit{standard set} of horoballs. This was first introduced, in a more general situation,  by Stratmann and Velani  \cite{SV}. Let $H_{\gamma}$ be a horoball tangential to the point $p$, and let $H_{g}$ be the image of $H_{\gamma}$ under the map $g\in G$. Note that if the map $g$ belongs to the stabiliser $G_p$ of $p$, which is the set $G_p:=\{g\in G:g(p)=p\}\cong\Z$, then the horoball $H_{g}$ is equal to $H_{\gamma}$. The horoball $H_h$ is independent of the choice of $h\in[g]:=G/G_p$, so $H_g$ is well-defined for $g\in G/G_p$. It is well known that the set $\{H_{g}:g\in G/G_p\}$ can be chosen in such a way that it is a pairwise disjoint collection of horoballs. This set is a standard set of horoballs for $G$. %From now on, let  $\{H_{g}:g\in G/G_p\}$ denote a fixed standard set of horoballs for $G$.

Let $s_\xi$ denote the hyperbolic half-ray between the origin and the point $\xi$ on $\S^1$. We will think of this ray as having an orientation, so that we travel from 0 towards $\S^1$. Define the \textit{top} of the standard horoball $H_{g}$ to be the first point on the boundary of $H_g$ reached whilst traveling along $s_{g(p)}$, that is,
\[
\tau_g:=s_{g(p)}\cap \partial H_{g}\cap \D^2.
\]
It was shown in \cite{SV} that the point $\tau_g$ lies a bounded distance away from the orbit of the origin under $G$. Examining the proof of this fact given there, they show that it is possible to choose  a set $\mathfrak{T}$ of coset representatives of $G/G_p$ in a geometric way, namely,  $g$ is in $\mathfrak{T}$ if the orbit point $g(0)$ lies in a $\rho$-neighbourhood of $\tau_g$, the top of the horoball $H_{g}$, where $\rho$ is the bound on the distance between the tops of standard horoballs and the orbit of 0. That is
\[
g\in\mathfrak{T}\Rightarrow d_h(\tau_g, g(0))\leq\rho.
\]
This is referred to  as the \textit{top representation} and from here on we will write $\{H_{g}:g\in\mathfrak{T}\}$ for a fixed standard set of horoballs for $G$ with top representation. Note that the choice for $g$ is not necessarily unique, but that this does not matter.

Before stating  our fist main theorem, we make the following definitions:
\begin{itemize}
  \item Let $\mathcal{L}(G)$ denote the set of all those $\xi\in L(G)$ with the property that the ray $s_\xi$ intersects infinitely many standard horoballs with top representation $H_{g_1}(\xi), H_{g_2}(\xi), H_{g_3}(\xi), \ldots$, which we always assume to be ordered according to their appearance when traveling from 0 to $\xi$.
  \item We call the distance traveled by a ray $s_\xi$ inside a standard horoball a \index{cusp excursion}\textit{cusp excursion}.  For each $\xi\in \mathcal{L}(G)$, let $d_n(\xi)$ denote the depth of the $n$-th cusp excursion, that is,
\[
d_n(\xi):=\max\{d(\eta,\partial H_{g_{n}}(\xi)):\eta\in s_\xi\cap \mathrm{Int} (H_{g_{n}}(\xi))\}.
\]
  \item For $\kappa>0$, let  $\mathcal{B}_\kappa$ denote the set of all those $\xi\in \mathcal{L}(G)$ with the property that the distance traveled between each cusp excursion is bounded by $\kappa$. In other words, $\mathcal{B}_\kappa$ is defined  to be
\[
{\mathcal  B}_{\kappa}(G): =  \{ \xi \in \mathcal{L}(G):  d(H_{g_{n}}(\xi),H_{g_{n+1}}(\xi) )< \kappa, \, \hbox{ for all } \,
n \in \N\}.
\]
  \item Now, for $\kappa, \tau>0$,  define the \textit{$(\tau, \kappa)$-Good set} by
\[
{\mathcal G}_{\tau,\kappa}(G) := \{\xi \in {\mathcal  B}_\kappa(G): d_{n}(\xi) > \log\tau, \, \hbox{ for all } \, n \in \N\}.
\]
%Then, finally, let the \index{$\tau$-Good set}\textit{$\tau$-Good set} be given by
%\[ {\mathcal G}_{\tau}(G) := \bigcup_{\kappa>0}{\mathcal G}_{\tau, \kappa}(G).\]

\end{itemize}

We then have the following theorem.

{\bf Theorem 1.} For all $\kappa>0$, we have that
\[
\lim_{\tau \to \infty} \dim_{H}\left({\mathcal G}_{\tau, \kappa}(G)\right) =
\frac12.\]

\begin{rem}
If the group $G$ is chosen to be $PSL_2(\Z)$, this result provides the corollary that if we define the set $F_N$ to be the set of all those numbers in $[0,1]$ with continued fraction representation containing only entries at least as large as $N$, that is, if $F_N:=\{x=[a_1(x), a_2(x), \ldots]:a_n(x)\geq N\text{ for all } n\in\N\}$, we have that
\begin{eqnarray}\label{origgood}
\lim_{N\to\infty}\dim_H(F_N)=\frac12.
\end{eqnarray}
Note that $\kappa$ does not play a role here, as the standard set of horoballs for the group $PSL_2(\Z)$ can be chosen to be the \textit{Ford circles} (the horoball at $\infty$ is the horizontal Euclidean straight line through the point $i$ and the images of this line are Euclidean circles with base point $p/q\in\mathbb{Q}$, in reduced form,  and radius $1/2q^2$) and there exists a global constant bound on the hyperbolic distance between neighbouring Ford circles.
The result  (\ref{origgood}) can also be obtained from Theorem 2 in the 1941 paper of I.J. Good \cite{Good} (hence the name for the $(\tau, \kappa)$-Good set), where he gives upper and lower bounds for the Hausdorff dimension of each set $F_N$. For the details of the beautiful connection between the  modular group $PSL_2(\Z)$ and the continued fraction expansion, the reader is referred to  Series \cite{series}. We also mention that Jaerisch and Kesseb\"ohmer \cite{JKarith} have obtained the precise asymptotic for $\dim_H(F_N)$. In addition to this, the recent paper of Jordan and Rams \cite{JorRams} contains results concerning Good-type sets.
\end{rem}

%\begin{rem}
%We mention a few well-known results which are of interest here. It is a result of Bishop and Jones \cite{BJ} (see also the paper of Stratmann \cite{sbj}), that for \textit{any} non-elementary Fuchsian group $G$, \[\dim_H(L_{ur}(G))=\dim_H(L_r(G))=\delta(G),\] where $L_{ur}(G)$ and $L_r(G)$ denote the uniformly radial and radial limit sets respectively. Beardon and Maskit \cite{bm} proved that  a Fuchsian group $G$ is geometrically finite if and only if $L(G)=L_r(G)\cup L_p(G)$, where $L_p(G)$ denotes the (countable) set of parabolic fixed points. Beardon \cite{ineq} also proved that if $G$ be a non-elementary geometrically finite Fuchsian group with at least one parabolic element, then we have that $\delta(G)>1/2$, where $\delta(G)$ denotes the exponent of convergence of the  Poincar\'{e} series of $G$. Combining these results, we see that for each geometrically finite Fuchsian group $G$ with at least one parabolic element, it follows that \[\dim_H(L(G))>1/2.\]

%\end{rem}

We are now ready to state our next  main result. First, define $t_n(\xi):=d_h(0, z_{g_n})+d_n(\xi)$, where $z_{g_n}$ is the point the ray from 0 to $\xi$ enters the $n$-th horoball (that is, the point just before the $n$-th cusp excursion begins). Then, for $\kappa>0$ and $\theta\in[0,1]$, define the \textit{ strict $(\theta, \kappa)$-Jarn\'{\i}k  set} $\mathcal{J}_{\theta, \kappa}^*(G)$ by setting
\[\mathcal{J}_{\theta, \kappa}^*(G):=\left\{\xi\in\mathcal{B}_\kappa:\lim_{n\to\infty}d_n(\xi)=\infty\text{ and }\limsup_{n\to\infty}\frac{d_n(\xi)}{t_n(\xi)}=\theta\right\}.
\]
%Define the \textit{strict $\theta$-Jarn\'{\i}k  set} to be \[\mathcal{J}_{\theta}^*(G):=\bigcup_{\kappa>0}\mathcal{J}_{\theta, \kappa}^*(G).\]
%We have the following theorem.

{\bf Theorem 2.}
For each $\kappa>0$, we have that
\[
\dim_H(\mathcal{J}_{\theta,\kappa}^*(G))=\frac12(1-\theta).
\]

\begin{rem}\label{onepara}
Although Theorems 1 and 2 are stated in terms of a Fuchsian group with one parabolic element, in fact this restriction is not really necessary. As can been seen in \cite{SV}, a standard set of horoballs with top representation can be chosen for any geometrically finite Fuchsian group with at least one parabolic point, that is, with any finite number of non-equivalent parabolic points. Since the proofs of our main theorems are based on defining covers of each $(\tau, \kappa)$-Good set and strict  $(\theta, \kappa)$-Jarn\'{\i}k set using these horoballs and increasing the number of balls with the same approximate size by a fixed finite number does not change any of our estimates, it still holds that analogous sets defined for cusp excursions into finitely many cusps have the same dimension. However, we choose to write the proofs for the case of one parabolic element for the sake of clarity.
\end{rem}

Our final main result is the following.
Let the $\beta$-strict-Jarn\'{\i}k level sets for the Patterson measure $\mu$ be defined by
\[ {\mathcal F}_{\beta}^* := \left\{\xi \in  L(G) :
  \limsup_{n \to \infty} \frac{\log \mu (b(\xi, e^{-t_n(\xi)}))}{-t_n(\xi)} = \beta \right\}.\]
We obtain the following theorem.

{\bf Theorem 3.}
\textit{
     For each $\beta \in [2\delta-1, \delta]$, we have that}
\[ \dim_{H} \left(
{\mathcal F}_{\beta}^* \cap  {\mathcal B}(G)
    \right) =  \frac12\cdot f_p(\beta),\]
    where $f_p(\beta):=(\beta-(2\delta-1))/(1-\delta).$

\begin{rem}
Theorems 2 and 3 are analogues of the results obtained by Stratmann in \cite{SJarnik1} and \cite{SJarnik2}. At the end of Section 4 we will explain this remark more fully. His results are for so-called Jarn\'{\i}k limit sets, which are defined similarly to our strict Jarn\'{\i}k sets but without the restriction to points in $\mathcal{B}_\kappa$. This is the reason for the word `strict' appearing in our notation.
\end{rem}

We assume that the reader is familiar with the construction and basic properties of Hausdorff dimension. For more details, a useful reference is \cite{Fal}. One particularly useful property that we will employ without further mention is that if $E\subseteq F\subseteq \R^d$, then $\dim_H(E)\leq \dim_H(F)$.
In the proofs of the above main theorems, we will also make use of the following well-known lemma, due to Frostman \cite{frostman}.
\begin{lem}\label{frost}
{\em (Frostman's Lemma.)} Let $F$ be a bounded subset of $\R^n$. Let $\mu$ be a finite Borel measure supported  on $F$, let $|U|$ denote the diameter of the set $U$ and suppose that for some $s>0$ there exist constants $c>0$ and $\delta>0$ with the property that
\[
\mu(U) \leq c|U|^s
\]
for all sets $U$ with $|U|\leq\delta$. Then $\mathcal{H}^s(F)\geq\frac{\mu(F)}{c}$ and so
\[
s\leq \dim_H(F).
\]
\end{lem}

\section{Preliminaries}

Let us first recall the  definition of the \textit{shadow map} $\Pi:\mathcal{P}(\D^2)\to \mathcal{P}(\S^1)$, which  is given by \[\Pi(A):=\{\xi\in\S^1:s_\xi\cap A\neq\emptyset\}. \]
 Also, if we have two functions $f,g:\R\to\R$ such that there exists a constant $c>1$ such that the inequality $c^{-1}g(x)\leq f( x) \leq cg(x)$ holds uniformly for $x\in\R$, then the functions $f$ and $g$ are said to be \index{comparability}\textit{comparable}, which we denote by $f\asymp g$. For future reference, if we wish to refer only to the right-hand side of the above inequality, we write $f\ll g$. % if there exists a constant $c\geq 1$ such that $c^{-1}y\leq x \leq cy$.
We obtain the following estimate for the size of the shadow of a standard horoball using basic hyperbolic geometry: %For every standard horoball with top representation from the set $\{H_{g}:g\in\mathfrak{T}\}$ we have that
\begin{eqnarray}\label{sizeball}
|\Pi(H_{g})|\asymp e^{-d(0, \tau_g)}.
\end{eqnarray}

\begin{rem}\label{shrem1}
Notice that combining the fact that the top of each standard horoball $H_g$ is within a constant distance of the orbit point $g(0)$  with the above estimate also yields that
\[
|\Pi(H_{g})|\asymp e^{-d(0, g(0))}.
\]
\end{rem}

%Recall the stabiliser $G_p$ of the parabolic point $p$, which is given by $G_p:=\{g\in G:g(p)=p\}$. For example, in the case of the map $z\mapsto z+1$ which fixes the point at infinity, the stabiliser is given by $G_{\{\infty\}}=\{g(z):=z+n:n\in\Z\}$.
Let $F_G$ be a fundamental region for $G$ with the property that one vertex of $F_G$ is equal to $p$. In the tessellation of hyperbolic space given by the region $F_G$, each map in the stabiliser of $p$ sends $F_G$ to a region that also has one vertex equal to $p$. %This structure is illustrated for the map $P:z\mapsto z+1$ in the upper half-plane and also for the Poincar\'{e} disc equivalent in Figure 2.1 below.
We will refer to the countably many copies of $F_G$ that  each have one vertex at a given point in $G(p)/G_p$ as \index{petals}\textit{petals}. In somewhat of an abuse of notation, we will use this word interchangeably to mean the arc of $\S^1$ enclosed by the two edges of the petal.

%\begin{figure}[htpb]
%\begin{center}
%\input{petals.pstex_t}\caption{Petals around the parabolic point at infinity on the boundary of $\mathbb{H}$ and the point $1$ on the boundary of $\D^2$.}
%\end{center}
%\end{figure}

Let us now recall the definition of the cross-ratio. For us (by which we mean that this is sometimes found differently in the literature), the \textit{cross-ratio} of four points $x, y, z, t$ in $\mathbb{H}\cup\R\cup\{\infty\}$ is given by
\[
[x, y, z, t]:=\frac{(x-y)(z-t)}{(y-z)(t-x)}.
\]

Now consider the following situation. %We have the following lemma.
Let $x$ and $y$ be two distinct points in the upper-half plane. Suppose that either $\mathrm{Re}(x)<\mathrm{Re}(y)$, or,
 if  $\ \mathrm{Re}(x)=\mathrm{Re}(y)$, suppose that $\mathrm{Im}(x)<\mathrm{Im}(y)$. Let $\xi$ and $\eta$ denote the start and end points of the oriented geodesic that joins $x$ to $y$ (see Figure 2.1).

\begin{figure}[htbp]
\begin{center}
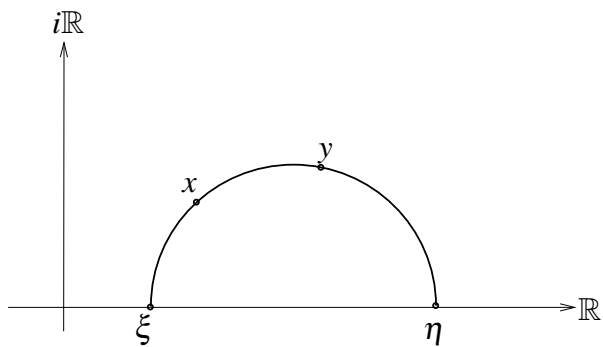\caption{The oriented geodesic through $x$ and $y$ with startpoint $\xi$ and endpoint $\eta$.}
\end{center}
\end{figure}

In this situation we have a a well known and extremely useful hyperbolic distance formula,  which is given in the following proposition.
\begin{prop}\label{crf}
Let $x, y, \xi$ and $\eta$ be as described above. Then
\[
d_{\mathbb{H}}(x,y)=\log([y, \xi, x, \eta]).
\]
\end{prop}

\begin{proof}
 Let $g(z):=(az+b)/(cz+d)\in PSL_2(\R)$. Then, using the fact that $g'(z)=1/(cz+d)^2$, it can be shown that the cross-ratio is $g$-invariant. In other words, for all $g\in PSL_2(\R)$ and all distinct points $x, y, z, t \in \mathbb{H}\cup \R\cup \{\infty\}$, we have that $[g(x), g(y), g(z), g(t)]=[x,y,z,t]$. Now define the map $g\in PSL_2(\R)$ by setting
\[
g(z):=\frac{(\xi-\eta)^{-1}(z-\xi)}{(\xi-\eta)^{-1}(z-\eta)}.
\]
This map sends $\xi$ to zero and $\eta $ to $\infty$, therefore it maps the points $x$ and $y$ to two points on the imaginary axis, say $ia$ and $ib$, respectively.
We  then have that
\begin{eqnarray*}
d(x,y)&=&d(g(x),g(y))=\log (b/a)\\&=&\log\frac{0-ib}{0-ia}=\log([ib, 0, ia, \infty])\\&=&\log([g(y), g(\xi), g(x), g(\eta)])=\log([y, \xi, x, \eta]).
\end{eqnarray*}

\end{proof}

%Also all the things about working out that the $n$-th petal is of size $1/n^2$.
Our next aim is to obtain an estimate of the size of the petals around the point $g(p)$ for all $g\in G$. This will be achieved with the help of the cross-ratio formula for distances given in Proposition \ref{crf}. Let us first consider the petals around the point $p$ itself. First of all, without loss of generality, suppose that the top of the horoball $H_\gamma$ is actually at 0; this will not alter any of the estimates by any more than a constant amount, due to the definition of the top representation. Then, let $r$ denote the rotation around 0 that moves $p$ to 1. This rotates the entire horoball $H_\gamma$. Now send this rotated picture into the upper half-plane, by way of the inverse of the Cayley transformation which maps $1\mapsto \infty$ and $0\mapsto i$. This procedure sends the map $\gamma$ to a parabolic element of $PSL_2(\R)$ in standard form, that is, in the form $z\mapsto z+\beta$, for $\beta\neq0$. Again without loss of generality, suppose that $\beta=1$. We will now make a particular hyperbolic distance estimate, which is illustrated in Figure 2.2. The notation``$a\asymp_+b$'' means that there exists a constant $K>0$ such that $a-K\leq b\leq a+K$.

\begin{figure}[htpb]
\begin{center}
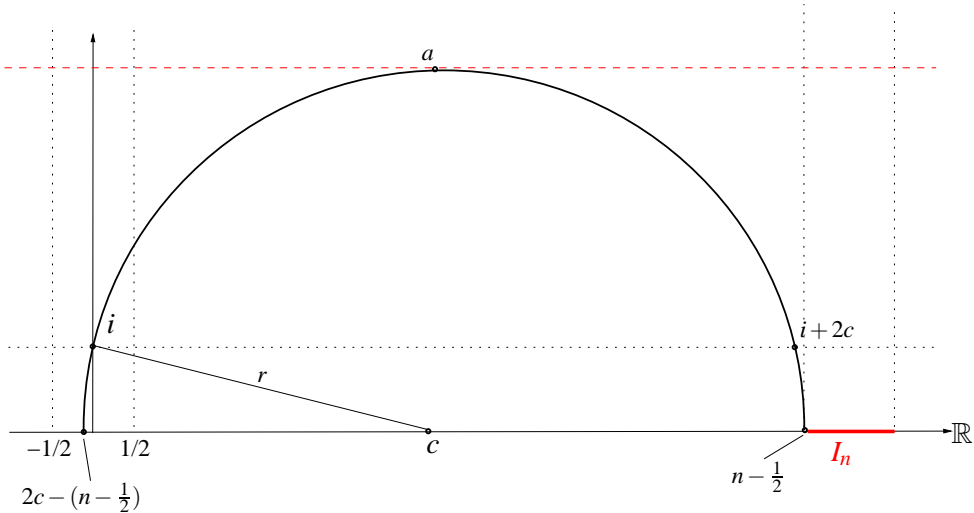\caption{Illustration of  the hyperbolic geodesic joining $i$ to $i+2c$.}
\end{center}
\end{figure}

\begin{lem}\label{forgoodjar}
As in Figure 2.2, with $c$ referring to the centre of the circle whose top half forms the geodesic in $\mathbb{H}$ joining $i$ to $n-1/2$, we have that
\[
d(i, i+2c)\asymp_+ 2\log n.
\]
\end{lem}

\begin{proof}
By Proposition \ref{crf}, we have that
\begin{eqnarray*}
d(i, i+2c)&=& \log([i+2c, 2c-(n-1/2), i, n-1/2])\\&=&\log\left(\frac{(i+2c-2c+(n-1/2))(i-(n-1/2))}{(2c-(n-1/2)-i)(n-1/2-(i+2c))}\right)
\\&=& \log\left(\frac{1+(n-1/2)^2}{(1+(2c-(n-1/2))^2)}\right).
\end{eqnarray*}
Let $r$ denote the radius of the semi-circle forming the geodesic joining $i$ to $n-1/2$. Then, on the one hand, we have that $r^2=1+c^2$, but on the other hand, $r=n-1/2-c$. So, after some elementary algebra, we obtain that $c=(4n^2-4n-3)/(4(2n-1))$. It follows that
\begin{eqnarray}\label{n2.1}
2c-(n-1/2)=\frac{-2}{2n-1}\ \text{ and so, }\ 1\leq 1+(2c-(n-1/2))^2\leq 5.
\end{eqnarray}
Also, $1+(n-1/2)^2=n^2-n+5/4$, which implies that provided $n\geq2$, we have
\begin{eqnarray}\label{n2.2}
\frac{n^2}{2}\leq1+(n-1/2)^2\leq n^2.
\end{eqnarray}
Combining (\ref{n2.1}) and (\ref{n2.2}) yields that
\[
\log(n^2)-\log 10\leq d(i, i+2c)\leq \log(n^2).
\]
This finishes the proof.

\end{proof}

Let $a$ refer to the point at the top of the semi-circle forming the geodesic between $i$ and $n-1/2$, as in Figure 2.2 above.
From the above lemma, it is immediately apparent from the fact that $a$ is the midpoint of geodesic segment between $i$ and $i+2c$ that
\[
d(i, a)\asymp_+ \log n.
\]

If we now return to the situation we were in before Lemma \ref{forgoodjar}, that is, if we return to the picture of the petals around the parabolic point $p\in\S^1$, we are now in a position to estimate the size of these petals. Let us denote the petal containing 0 by $I_0$ and then call the petals $(I_n)_{n\in\N}$ in sequence as they move around $\S^1$ to the point $p$. We are being a little vague here, because there are actually two sequences of petals clustering down to $p$, but since they are completely symmetric there is no real problem. We have for the hyperbolic distance between 0 and the point $r^{-1}\circ\Phi(i+2c)$, where $r$ is the rotation bringing $p$ to 1 and $c$ is as in Lemma \ref{forgoodjar}, that
\[
d(0,  r^{-1}\circ\Phi(i+2c))\asymp_+ \log (n^2).
\]
We also have, where $a$ is as above,  that
\[
d(0, r^{-1}\circ\Phi(a))\asymp_+ \log n.
\]
Note that $r^{-1}\circ\Phi(a)$ lies on a horoball at $p$ whose shadow contains all the petals $I_k$ for $k\geq n$. We can calculate that the hyperbolic distance from 0 to the top of this horoball is also comparable to $\log n$.  From this and from Proposition \ref{sizeball} it immediately follows that the size of the shadow of this horoball is comparable to $1/n$. Given that this holds for every $n\in\N$, we finally obtain that
\begin{eqnarray}\label{petalsize}
|I_n|\asymp \frac1{n^2}.
\end{eqnarray}

\vspace{2mm}

Let us now consider an arbitrary horoball $H_g$ from the standard set with top representation $\{H_{g}:g\in\mathfrak{T}\}$. Then $H_g=g(H_\gamma)$ and we will assume, again without loss of generality (in view of the choice of horoballs with top representation), that $\tau_g=g(0)$. As already mentioned, each image of the parabolic point $p$ has the same petal structure. It is then straightforward to calculate in a similar way to that above that if $I^{(g)}_n$ denotes the $n$-th petal around $g(p)$, we have that
\begin{eqnarray}\label{ineq}
\left|I^{(g)}_n\right|\asymp e^{-d(0, g(0))}\cdot\frac1{n^2}.
\end{eqnarray}

For Section \ref{final}, we will need the Patterson measure, which we will denote by $\mu$. Again, we assume that the reader is familiar with the definition and basic properties. For the purposes of this paper, all that is necessary is that the Patterson measure is a non-atomic probability measure supported on the limit set $L(G)$, it is $\delta$-conformal and, further, there exists a uniform estimate for the $\mu$-measure of balls in $\mathbb{S}^1$ centred around limit points. This estimate, called the \textit{Global Measure Formula}, was derived in \cite{SV} (see also \cite{Sulivan}). Here, we state the formula in only as much generality as we need, that is, only for a Fuchsian group, rather than a Kleinian group. In order to state the formula, we require the following notation. We define $b(\xi_t)$ to be the intersection of $\mathbb{S}^1$ with the disc whose boundary is  orthogonal to $\mathbb{S}^1$ and which intersects the ray $s_\xi$ orthogonally at $\xi_t$. So $b(\xi_t)$  is an arc of $\mathbb{S}^1$ with radius comparable to $e^{-t}$.  Further, define $k(\xi_t)$ to be equal to 1 if $\xi_t$ is inside some standard horoball $H_{g}$ and  let $k(\xi_t)$ be equal to $\delta$ otherwise.

\begin{lem}(Global Measure Formula). Let $G$ be a non-elementary, geometrically finite Fuchsian group with parabolic elements. If $\xi\in L(G)$ and $t$ is positive, then
\[
\mu(b(\xi_t))\asymp e^{-t\delta}e^{-(\delta-k(\xi_t))\Delta(\xi_t)}.
\]
\end{lem}

\section{Good sets}
In this section, we give the proof of our first main result.

{\em Proof of Theorem 1.} First let $\kappa>0$ be given and let us consider $\dim_{H}\left({\mathcal G}_{\tau, \kappa}(G)\right)$. We begin with the upper bound. Notice that the set ${\mathcal G}_{\tau, \kappa}(G)$ can be covered by any of the families
\[
\{\Pi(H_{g_1}(\xi)):\xi\in {\mathcal G}_{\tau,\kappa}(G)\},
\]
%\[
%\{\Pi(H_{g_2}(\xi)):\xi\in {\mathcal C}_{\tau,\kappa}(G)\},
%\]
\[\vdots\]
\[\left\{\Pi(H_{g_k}(\xi)):\xi\in {\mathcal G}_{\tau,\kappa}(G)\right\}
\]
\[\vdots\]
Eventually, if $k$ is chosen sufficiently large, the cover $\{\Pi(H_{g_k}(\xi)):\xi\in {\mathcal G}_{\tau,\kappa}(G)\}$ will consist of sets of diameter less than any fixed positive $\delta$. Now, let $s=\frac12(1+\var_\tau)$, where $\var_\tau$ is chosen such that $\var_\tau<1$ and $1/\log(\lfloor\tau\rfloor-1)\leq\var_\tau/(-\log\var_\tau)$. Then,
note that for each $\xi \in \mathcal{B}_\kappa(G)$, the shadows of the standard horoballs intersected by the ray $s_\xi$ form a nested sequence of intervals of $\mathbb{S}^1$ which cluster down to the point $\xi$. We can associate to the sequence of horoballs a sequence of positive integers $a_1(\xi), a_2(\xi), \ldots$ with the property that
\[
\log(a_n(\xi))\leq d_n(\xi)<\log(a_n(\xi)+1).
\]
%Note that by the discussion at the end of Section \ref{sthoro}, this means that (up to some constant of comparability) we have that $\xi$ lies in the $a_n$-th petal around the point $g_n(p)$.
Consequently, for any $\xi\in {\mathcal  B}_{\kappa}(G)$ and any $n\in\N$, by Proposition \ref{sizeball} above, we have the following estimate.
\begin{eqnarray}\label{eqn1}
\ \ \ \ \frac{1}{ e^{(n+1)\kappa}((a_1(\xi)+1)\cdots(a_n(\xi)+1))^2}\ll|\Pi(H_{g_n}(\xi))|\ll  \frac{1}{(a_1(\xi)\cdots a_n(\xi))^2}.
\end{eqnarray}

It then follows, by the choice of $\var_\tau$,  that
\begin{eqnarray*}
\mathcal{H}^s_\delta({\mathcal C}_{\tau,\kappa}(G))&\leq& \sum_{\xi\in{\mathcal C}_{\tau,\kappa}(G)}|\Pi(H_{g_k}(\xi))|^s\\
&\ll&\sum_{a_1\geq\lfloor\tau\rfloor}\frac1{a_1^{2s}}\left(\sum_{a_2\geq\lfloor\tau\rfloor}\frac1{a_2^{2s}}\cdots\left(
\sum_{a_k\geq\lfloor\tau\rfloor}\frac1{a_k^{2s}}\right)\cdots\right)\\&\ll&\left(\int_{\lfloor\tau\rfloor}^\infty\frac{1}{x^{2s}}\ dx \right)^k\\
&=&\left(\frac{1}{\var_\tau(\lfloor\tau\rfloor-1)^{\var_\tau}}\right)^k<1.
\end{eqnarray*}
As this is true for any arbitrary $\delta>0$, it follows that $\mathcal{H}^s({\mathcal C}_{\tau,\kappa}(G))$ is finite and consequently that $\dim_H({\mathcal C}_{\tau,\kappa}(G))\leq s=\frac{1}{2}(1+\var_\tau)$. If  we then choose $\var_\tau$ in such a way that $\lim_{\tau\to\infty}\var_\tau=0$, we obtain the desired upper bound, %\footnote{This is certainly possible. As an example, for each $\tau\in[3, 5)$, let $\var_{\tau}=3/4$ and then for each $\tau\in[n^n+1, (n+1)^{(n+1)}+1)$, let $\var_\tau=1/n$, for $n\geq2$. This choice satisfies the conditions on $\var_\tau$. }
that is,
\[
\lim_{\tau\to\infty}\dim_{H}({\mathcal C}_{\tau,\kappa}(G))\leq\frac12.
\]

For the lower bound, again fix $\tau\geq3$ and $\kappa>0$. We first describe a subset of the set in question and then employ Frostman's Lemma to estimate from below the dimension of this subset. So, to that end, choose $\tau'$ to satisfy the equation
\[
\sum_{i=\lfloor\tau\rfloor}^{\lfloor\tau'+1\rfloor}\frac1{i+1}>e^{\kappa/2}.
\]
Denote this sum by $S$. Let $\mathcal{C}_{\tau', \kappa}(G)$ be the set
\[
\mathcal{C}_{ \tau', \kappa}(G):=\{\xi\in\mathcal{B}_\kappa:\log{\tau}< d_n(\xi)\leq \log{\tau'}\text{ for all }n\in\N\}.
\]
Let $\nu$ be a measure supported on the limit set $L(G)$ with the property that
\[
\nu\left(\Pi(H_{g_k}(\xi))\right)=\frac1{S^k}\cdot\frac{1}{(a_1(\xi)+1)\cdots(a_k(\xi)+1)}\leq \frac{c_1\ e^{((k+1)\kappa)/2}}{S^k}|\Pi(H_{g_k}(\xi))|^\frac12,
\]
where $c_1$ is a constant. Note that by the choice of $\tau'$, the term $c_1\left(e^{\kappa/2}/{S}\right)^ke^{\kappa/2}$ is simply another constant, say $c_2$. Now, let $\xi\in\mathcal{C}_{\tau', \kappa}$ and let $r>0$. Then choose the first $k$ such that the shadow of the  $(k+1)$th level horoball $H_{g_{k+1}}(\xi)$ is at most equal to $r$, that is, choose $k$ such that
\[
|\Pi(H_{g_{k+1}}(\xi))|\leq r<|\Pi(H_{g_{k}}(\xi))|.
\]
Note that for each $k\in\N$ there can  only be  a fixed finite number of petals around any point $g_k(p)$ in which it is possible for a point in the set $\mathcal{C}_{ \tau', \kappa}(G)$ to end up. Therefore there can only be a fixed finite number of horoballs that $B(\xi, r)$ could possibly intersect at each level $k$ and furthermore, each of these shadows has comparable $\nu$-measure, so,   without loss of generality we suppose that
\[
 \Pi(H_{g_{k+1}}(\xi))\subset B(\xi, r) \subset \Pi(H_{g_{k}}(\xi)).
\]
We now need to compare  the sizes of the above shadows. Directly from Proposition \ref{sizeball}, we obtain that
\[
\frac{|\Pi(H_{g_{k}}(\xi))|}{|\Pi(H_{g_{k+1}}(\xi))|}\asymp e^{d(0, \tau_{g_{k+1}}(\xi))-d(0, \tau_{g_{k}}(\xi))}.
\]
It can be shown, via the cross-ratio distance formula again, that if $z_{g_k}$ denotes the point that the geodesic segment joining 0 and $\tau_{g_{k+1}}$ first intersects the horoball $H_{g_k}(\xi)$, then $d(0, \tau_{g_k})\asymp d(0, z_{g_k})$. It follows that
\[
d(0, \tau_{g_{k+1}})\asymp d(0, \tau_{g_{k}})+2d_k(\xi)+\kappa.
\]
Therefore, keeping in mind that $d_k(\xi)\leq \log \tau'$, we obtain that
\[
d(0, \tau_{g_{k+1}}(\xi))-d(0, \tau_{g_{k}}(\xi))\ll \log \tau'+\kappa.
\]
Finally, then, since $\tau'$ is fixed for each $\tau$, there exists another constant $c_3$ such that $|\Pi(H_{g_{k}}(\xi))|\leq (c_3)^2|\Pi(H_{g_{k+1}}(\xi))|$. Then,
%Noting that $d(0, \tau_{g_{k+1}}(\xi))-d(0, \tau_{g_{k}}(\xi))\leq e^{\kappa+2\log(\tau')}$ and since $\tau'$ is fixed for each $\tau$, it follows that there is another universal constant $c_3$ such that $|\Pi(H_{g_{k}}(\xi))|\leq c_3|\Pi(H_{g_{k+1}}(\xi))|$. Then,
\begin{eqnarray*}
\nu(B(\xi, r)&\leq&\nu(\Pi(H_{g_{k}}(\xi)))\leq c_2|\Pi(H_{g_{k}}(\xi))|^\frac12\\
&\leq&c_2c_3|\Pi(H_{g_{k+1}}(\xi))|^\frac12\\&\leq&c_4\cdot r^\frac12,
\end{eqnarray*}
where $2c_4:=c_2c_3$. Thus, by Frostman's Lemma, we have that $\dim_H(\mathcal{G}_{\tau', \kappa}(G))\geq 1/2$ and therefore $\dim_H(\mathcal{G}_{\tau, \kappa}(G)\geq1/2$, too. Combining this and the upper bound obtained previously, we have that
\[
\lim_{\tau\to\infty}\dim_H\left(\mathcal{G}_{\tau, \kappa}(G)\right)=\frac12.
\]
Therefore, as the choice of $\kappa>0$ was arbitrary, the proof is finished.

\bewend

\section{Strict Jarn\'{\i}k Sets}\label{jarstar}

In preparation for the proof of Theorem 2, we will first prove the following lemma. Before stating the lemma, we need some notation. So let $s:=(s_n)_{n\in\N}$ be a sequence of positive integers such that $\lim\limits_{n\to\infty}s_n=\infty$ and
\[
%\liminf_{n\to\infty}\frac{\log(s_1\ldots s_n)}{2\log(s_1\ldots s_{n})+\log s_{n+1}}=\omega\in\left[0,\textstyle\frac12\right].
\limsup_{n\to\infty}\frac{\log(s_n)}{2\log(s_1\ldots s_{n-1})}=\omega.
\]
Recall the definition of $d_n(\xi)$ from the introduction. Then, with a constant $\kappa>0$ and a positive integer $N>3$, define the set $F_{s, N, \kappa}(G)$ to be
\[
F_{s, N, \kappa}(G):=\{\xi\in\mathcal{B}_\kappa: \log s_n\leq d_n(\xi)< \log{Ns_n}\ \text{ for all } n\in\N\}.
\]
Further, define \[F_{s, \kappa}(G):=\bigcup_{N>3}F_{s, N, \kappa}.\]
We then have the following result.

\begin{lem}\label{chinese}For each $\kappa>0$, we have that
%For $(s_n)_{n\in\N}$ a sequence of positive real numbers as described above, we have that
\[
\dim_H(F_{s, \kappa}(G))=\frac{1}{2(1+\omega)}.
\]
\end{lem}

Before starting the proof of the lemma, let us gather a few useful facts that will be required in the proof.

\begin{itemize}
  \item Since $\lim\limits_{n\to\infty}s_n=\infty$, it follows that $\lim\limits_{n\to\infty}\log s_n=\infty$ and thus that
\begin{eqnarray}\label{star}
\lim_{n\to\infty}\frac{\log(s_1\ldots s_n)}{n}=\infty.
\end{eqnarray}
  \item Define \[\rho:=\liminf_{n\to\infty}\frac{\log (s_1\ldots s_n)}{\log( (s_1\ldots s_n)^2s_{n+1})}=\frac1{2(1+\omega)}.\]
Then, for all $K>0$, we have that
\begin{eqnarray}\label{starstar}
\liminf_{n\to\infty}\frac{\log (s_1\ldots s_n)}{\log ((K^ns_1\ldots s_n)^2s_{n+1})}=\rho.
\end{eqnarray}
Indeed, if we write
\[
\frac{\log (s_1\ldots s_n)}{\log ((K^ns_1\ldots s_n)^2s_{n+1})}=\frac{\log (s_1\ldots s_n)}{\log ((s_1\ldots s_n)^2s_{n+1})}\cdot\frac{1}{1+\frac{2n\log K}{\log ((s_1\ldots s_n)^2s_{n+1})}},
\]then the statement in (\ref{starstar}) follows immediately from (\ref{star}) and the simple analytical fact that if $(b_n)_{n\in\N}$ is a sequence with each $0<b_n<1$  and $\lim_{n\to\infty}b_n=1$ and if $(a_n)_{n\in\N}$ is a sequence of positive real numbers, then $\limsup_{n\to\infty}\  a_nb_n=\limsup_{n\to\infty} \ a_n$. We also note here that this analytical observation will  be used repeatedly in the proof of Theorem 2.
  \item For all $K>0$, all $\rho'<\rho$ and sufficiently large $n\in\N$, we have that
\begin{eqnarray}\label{starx3}
\frac{1}{s_1\ldots s_n}\leq \left(\frac{1}{K^{2n}(s_1\ldots s_n)^2s_{n+1}}\right)^{\rho'}.
\end{eqnarray}
This follows directly from (\ref{starstar}) and the definition of the lower limit.
\end{itemize}

\begin{proof}[Proof of Lemma \ref{chinese}.] Fix a sequence $(s_n)_{n\in\N}$, $\kappa>0$ and $N>3$ as in the statement of the lemma.
 Again, we first establish the upper bound, then the lower bound. To begin, just as in the proof of Theorem 1, we can associate to each point $\xi\in F_{s, \kappa, N}(G)$ a sequence of positive integers $(a_n(\xi))_{n\geq1}$, where $a_n(\xi)$ is determined by
 \[
\log(a_n(\xi))\leq d_n(\xi)< \log(a_n(\xi)+1).
\]
Notice that this implies that the point $\xi$ lies, up to some constant of comparability, in the $(a_n(\xi))$-th petal around the point $g_n(p)$, for each $n\in\N$. We therefore have, from Proposition \ref{sizeball} and the  extra information given by the sequence $(s_n)_{n\geq1}$, that
\begin{eqnarray*}
\frac{1}{e^{(n+1)\kappa}(N^ns_1\ldots s_n)^2}\ll|\Pi(H_{g_{n+1}}(\xi))|\ll\frac{1}{( a_1(\xi)\ldots a_n(\xi))^2}\leq\frac{1}{(s_1\ldots s_n)^2}.
\end{eqnarray*}
For each positive integer $n$, define the ``shrunken'' horoball $\widetilde{H}_{g_{n}}(\xi)$ to be the horoball with base point $g_{n}(p)$ and top $\tilde{\tau}_{g_{n}}$ given by
\[
d(0, \tilde{\tau}_{g_{n}})=d(0, {\tau}_{g_{n}})+\log s_{n}.
\]
It follows immediately that
\begin{eqnarray}\label{label}
\frac{1}{e^{(n+1)\kappa}(N^ns_1\ldots s_n)^2s_{n+1}}\ll|\Pi(\widetilde{H}_{g_{n+1}}(\xi))|\ll\frac{1}{(s_1\ldots s_n)^2s_{n+1}}.
\end{eqnarray}

We will now provide the upper bound. This is based on covers of arbitrarily small diameter for the set $F_{s, \kappa, N}(G)$, which are  given by the shrunken horoballs defined above. First, let us make the observation that if $\xi\in F_{s, \kappa, N}(G)$, it follows that $\log s_n\leq d_n(\xi)<\log (Ns_n)$  and thus that $\xi$ could lie in any of the petals around $g_n(p)$ from the $s_n$-th up to the $(Ns_n)$-th. So, there are $c(N-1)s_n$ shrunken horoballs in the $n$-th layer that the point $\xi$ could lie in the shadow of, where $c$ is the fixed constant number of these horoballs that have their base point in any given petal.

Now, by the definition of $\rho$ given above, it follows that if we let $\rho'\in(\rho, 3\rho)$, we have for all sufficiently large $n$ that
\[
\frac{\rho'-\rho}{2}\leq \frac{\log (s_1\ldots s_n)}{\log((s_1\ldots s_n)^2s_{n+1}) }.
\]
Consequently, from the identity $(1/b)^{\log a/\log b}=1/a$, we obtain the inequality
\[
\left(\frac1{(s_1\ldots s_n)^2s_{n+1}}\right)^{(\rho'-\rho)/2}< \left(\frac{1}{(s_1\ldots s_n)^2s_{n+1}}\right)^{\log (s_1\ldots s_n)/\log ((s_1\ldots s_n)^2s_{n+1})}=\frac{1}{s_1\ldots s_n}.
\]
In other words,
\[
s_1\ldots s_n\leq ((s_1\ldots s_n)^2s_{n+1})^{(\rho'-\rho)/2}.
\]
Then we observe by equation (\ref{star}) that for all sufficiently large $n$ we have that $\log(N-1)<\log (s_1\ldots s_n)/n\ $ (since the left-hand side is simply a constant depending only on $N$). Therefore, for sufficiently large $n$,
\begin{eqnarray}\label{no1}
(N-1)^n<((s_1\ldots s_n)^2s_{n+1})^{(\rho'-\rho)/2}.
\end{eqnarray}
Also directly from the definition of $\rho$, for any $\rho'>\rho$, there exists a sequence $(n_k)_{k\in\N}$ with the property that
\[
\frac{\log (s_1\ldots s_{n_k})}{\log ((s_1\ldots s_{n_k})^2s_{{n_k}+1})}\leq \frac{\rho'+\rho}{2}, \  \text{ for all }k\geq1.
\]
Hence, on rearranging the above expression, we obtain that \begin{eqnarray}\label{no2}s_1\ldots s_{n_k}\leq  ((s_1\ldots s_{n_k})^2s_{{n_k}+1})^{(\rho'+\rho)/{2}}.\end{eqnarray} Consequently, if we neglect any terms of the sequence $(n_k)_{k\in\N}$ that are too small and rename the sequence accordingly, we have on combining (\ref{no1}) and (\ref{no2}) that for all $k\geq1$ and any $\rho'>\rho$,
\[
(N-1)^{n_k} s_1\ldots s_{n_k}<((s_1\ldots s_{n_k})^2s_{{n_k}+1})^{\rho'}.
\]
Thus, from (\ref{label}) and the above inequality, we infer that
\begin{eqnarray*}
\mathcal{H}^{\rho'}(F_{s,\kappa, N}(G))&\leq& \liminf_{k\to \infty}\sum_{{ s_m\leq a_m(\xi)< Ns_m}\atop{1\leq m \leq n_k}}|\Pi(\widetilde{H}_{g_{n_k+1}}(\xi))|^{\rho'}\\&\ll&
\left((N-1)^{n_k}s_1\ldots s_{n_k}\right)\left((s_1\ldots s_{n_k})^2s_{{n_k}+1}\right)^{-\rho'}\leq1.
\end{eqnarray*}
Hence, as this is true for all $\rho'\in(\rho, 3\rho)$, it follows that \[\dim_H(F_{s, \kappa,N}(G))\leq \rho=1/(2(1+\omega)).\]

For the lower bound, we will again use Frostman's Lemma. To that end, define, in the usual way as a weak limit of a sequence of finite Borel measures (see e.g. \cite{Fal} and \cite{bill}), a measure $m$ on the limit set $L(G)$ with the property that
\begin{eqnarray*}
m(\Pi(H_{g_{k+1}}(x))=\frac{1}{a_1(x)\ldots a_k(x)}.
\end{eqnarray*}
Let $\xi\in F_{s, \kappa,N}(G)$. Then for each small enough $r>0$, we can find a unique $k$ such that
\begin{eqnarray*}
|\Pi(\widetilde{H}_{g_{k+1}}(\xi))|\leq r <|\Pi(\widetilde{H}_{g_{k}}(\xi))|.
\end{eqnarray*}
We consider two further possibilities. Either,
\begin{eqnarray}\label{case1}
|\Pi(\widetilde{H}_{g_{k+1}}(\xi))|\leq r <|\Pi({H}_{g_{k+1}}(\xi))|,
\end{eqnarray}
or,
\begin{eqnarray}\label{case2}
|\Pi({H}_{g_{k+1}}(\xi))|\leq r <|\Pi(\widetilde{H}_{g_{k}}(\xi))|.
\end{eqnarray}

\vspace{3mm}

First note that by inequality (\ref{starx3}),  if we let $\rho'<\rho$, then there exists $n_0$ such that for all $n\geq n_0$, we have that
\begin{eqnarray}\label{liminfdefn}
\frac{1}{s_1\ldots s_{n}}\leq \left(\frac{1}{(e^{\kappa})^{n+1}(N^{n}s_1\ldots s_{n})^2s_{n+1}}\right)^{\rho'}.
\end{eqnarray}
Now, suppose that we are in the situation of (\ref{case1}). Choose $r$ so that $k+1\geq n_0$. In order to estimate the $m$-measure of the ball $B(\xi, r)$, we must first identify the number of shadows of standard horoballs in the $(k+1)$-th layer that said ball can intersect. Since there are a fixed number in each petal, it suffices to calculate the number of petals around $g_k(p)$ that the ball $B(\xi, r)$ can intersect. First of all, note that for large enough $k$, $B(\xi, r)$ cannot extend further than the $(a_k(\xi)-1)$-th petal, because the petals are decreasing in size. To finish the proof that $B(\xi, r)$, with small enough $r$,  can only intersect a finite number of petals around $g_k(p)$, we must show that these petals are not shrinking too fast. It suffices to show that where $c$ comes from the  comparability given in (\ref{ineq}), there exists a $k_0\in\N$ such that if $k>k_0$ there exists $M$ such that
\begin{eqnarray}\label{ineq2}
\sum_{i=1}^M \frac1{(a_k(\xi)+i)^2}\geq \frac{c^2}{(a_k(\xi))^2}.
\end{eqnarray}
This follows immediately from the fact that the $a_k$s are increasing and that $\sum_{i=1}^\infty k^2/(k+i)^2>k-1 $ for all $k\in\N$. %By rearranging (\ref{ineq2}) and, to shorten the notation, writing simply $a_k$ for $a_k(\xi)$, it is equivalent to show that there exists a constant $M$ such that
%\[
%a_k^2 \sum_{i=1}^M\prod_{j=1\atop j\neq i}^M(a_k+j)^2
%\geq c^2\prod_{j=1}^M(a_k+j)^2.
%\]
%The leading term of the left-hand side of the above inequality is $M\cdot a_k^{2M}$, whereas the leading term of the right-hand side is $c^2\cdot a_k^{2M}$. Therefore, provided that $M>c^2$, the claim is proved. 
Consequently, $B(\xi, r)$ can only intersect a fixed finite number of petals around $g_k(p)$ for each sufficiently large $k\in\N$ and can thus only intersect a fixed finite number of shadows of standard horoballs in each layer. It follows, via (\ref{liminfdefn}), (\ref{label}) and (\ref{case1}), that for any $\rho'<\rho$ we have
\begin{eqnarray*}
m(B(\xi, r))&\ll& m\left(\Pi(H_{g_{k+1}}(\xi))\right)\\
&=&\frac{1}{a_1(\xi)\cdots a_{k}(\xi)}\leq \frac{1}{s_1\cdots s_k}\\
&\leq& \left(\frac{1}{e^{\kappa(k+1)}(N^ks_1\cdots s_k)^2s_{k+1}}\right)^{\rho'}\\&\ll&r\ ^{\rho'}.
\end{eqnarray*}

If we are in the situation of (\ref{case2}), it is clear, by similar reasoning to that above, that $B(\xi, r)$ cannot intersect more than two petals in the $k$-th layer, which means that there is again only a fixed finite number of shadows of shrunken horoballs that $B(\xi, r)$ can intersect. In addition to this, a maximum of $2r(\kappa N^2)^{k}(s_1\ldots s_{k})^2$ of the $(k+1)$-th layer shadows of standard horoballs are intersected by $B(\xi, r)$. So, denoting by $\Pi(\widetilde{H}_{g_{k}})$ and $\Pi(H_{g_{k+1}})$ the largest possible shadow in layer $k$ and $k+1$ respectively, we have that
\begin{eqnarray*}
m(B(\xi, r))&\ll&\min\{2m(\Pi(\widetilde{H}_{g_{k}})), 2r(\kappa N^2)^{k}(s_1\cdots s_{k})^2m(\Pi(H_{g_{k+1}}))\}\\
&\leq &\frac{2}{s_1 \cdots s_k}\min\{1, (\kappa N^2)^{k+1}(s_1\cdots s_{k})^2s_{k+1}r \}
\end{eqnarray*}
and, using equation (\ref{liminfdefn}) and the fact that  $\min\{a, b\}\leq a^{1-s}b^s$ for any $0<s<1$, it follows that for $\rho'<\rho$ we have
\begin{eqnarray*}
m(B(\xi, r))&\leq&2\left(\frac{1}{(\kappa N^2)^{k}(s_1\ldots s_{k})^2s_{k+1}}\right)^{\rho'}\cdot \left(\left(\kappa N^2\right)^{k+1}(s_1\cdots s_{k})^2s_{k+1}r\right)^{\rho'}\\
&=&2\kappa N^2 r^{\rho'}.
\end{eqnarray*}
Thus, in each case, on applying Frostman's Lemma and letting $\rho'$ tend to $\rho$, we obtain that
\[
\dim_H(F_{s, N, \kappa})\geq \rho:=\frac{1}{2(1+\omega)}.
\]
Combining this with the previously obtained upper bound yields $\dim_H(F_{s, N, \kappa})=1/2(1+\omega)$, for all $N>3$ and all $\kappa>0$. Thus, as the Hausdorff dimension is countably stable (see \cite{Fal}), we have that for all $\kappa>0$
\[
\dim_H(F_{s, \kappa})=\frac{1}{2(1+\omega)}.
\]

\end{proof}

We are now in a position to prove Theorem 2.

\begin{proof}[Proof of Theorem 2.]
Fix $\kappa>0$.  The first step of the proof is to show that the condition $\limsup\limits_{n\to\infty}\frac{d_n(\xi)}{t_n(\xi)}=\theta$ is equivalent to the condition that \[\limsup_{n\to\infty}\frac{d_n(\xi)}{2(d_1(\xi)+\cdots+d_{n-1}(\xi))}=\frac{\theta}{1-\theta}.\]
%We use this observation to infer that if $\xi\in F_{\frac{\theta}{1-\theta}, N, \kappa}(G)$, for some $N>2$, then $\xi\in \mathcal{J}^*_{\theta, \kappa}(G)$. This will provide an upper bound for the sought-after Hausdorff dimension. Finally, in order to obtain the lower bound, we let $\xi \in \mathcal{J}^*_{\theta, \kappa}(G)$ and construct a sequence $(\hat{s}_n)_{n\in\N}$ satisfying the properties $\lim\limits_{n\to\infty} \hat{s}=\infty$ and $\limsup\limits_{n\to\infty}\frac{\log \hat{s}_n}{2\log(\hat{s}_1\ldots \hat{s}_{n-1})}=\frac{\theta}{1-\theta}$ and  such that $\xi \in F_{\frac{\theta}{1-\theta}, 3, \kappa}(G)$.
In order to do this, we begin by claiming that
\begin{eqnarray}\label{6.10.2}
\limsup_{n\to\infty}\frac{d_n(\xi)}{t_n(\xi)}=\theta\ \Leftrightarrow\ \limsup_{n\to\infty}\frac{d_n(\xi)}{d(0, z_{g_n})}=\frac{\theta}{1-\theta}.
\end{eqnarray}
Indeed, if $\theta>0$, we have that
\begin{eqnarray}\label{forstrictjar}
\theta=\limsup_{n\to\infty}\frac{d_n(\xi)}{t_n(\xi)}=\limsup_{n\to\infty}\frac{d_n(\xi)}{d(0, z_{g_n})+d_n(\xi)}=\frac{1}{1+\liminf\limits_{n\to\infty}\frac{d(0, z_{g_n})}{d_n(\xi)}}.
\end{eqnarray}
Therefore,
\[
\limsup_{n\to\infty}\frac{d_n(\xi)}{d(0, z_{g_n})}=\frac{1}{\frac1\theta -1}=\frac{\theta}{1-\theta}.
\]
On the other hand, if $\theta=0$, we have from (\ref{forstrictjar}) that
\[
\liminf_{n\to\infty}\frac{d(0, z_{g_n})}{d_n(\xi)}=\infty\ \Rightarrow\ \limsup_{n\to\infty}\frac{d_n(\xi)}{d(0, z_{g_n})}=0.
\]
Thus, since these arguments work equally well backwards, the claim in (\ref{6.10.2}) is proved.
Next, notice that \begin{eqnarray}\label{6.10.0}\limsup\limits_{n\to\infty}\frac{d_n(\xi)}{d(0, z_{g_n})}=\frac{\theta}{1-\theta} \Leftrightarrow\limsup\limits_{n\to\infty}\frac{d_n(\xi)}{2(d_1(\xi)+\cdots+d_{n-1}(\xi))}=\frac{\theta}{1-\theta}.\end{eqnarray} The reason for this is that we have \[2(d_1(\xi)+\cdots+d_{n-1}(\xi))\leq d(0, z_{g_n})\leq n\kappa+2(d_1(\xi)+\cdots+d_{n-1}(\xi)),\] so
\begin{eqnarray}\label{6.10.1}
\frac{d_n(\xi)}{n\kappa+2(d_1(\xi)+\cdots+d_{n-1}(\xi))}\leq \frac{d_n(\xi)}{d(0, z_{g_n})}\leq\frac{d_n(\xi)}{2(d_1(\xi)+\cdots+d_{n-1}(\xi))}.
\end{eqnarray}
Then, from the second inequality in (\ref{6.10.1}), it is immediate that
\[
\frac{\theta}{1-\theta}\leq \limsup\limits_{n\to\infty}\frac{d_n(\xi)}{2(d_1(\xi)+\cdots+d_{n-1}(\xi))}.
\]
Rewriting the first inequality from (\ref{6.10.1}), we obtain that
\[
\frac{d_n(\xi)}{2(d_1(\xi)+\cdots+d_{n-1}(\xi))}\cdot \frac{1}{1+\frac{n\kappa}{2(d_1(\xi)+\cdots+d_{n-1}(\xi))}}\leq \frac{d_n(\xi)}{d(0, z_{g_n})}.
\]
Consequently, recalling the fact that $\lim_{n\to\infty}d_n(\xi)=\infty$, we also obtain the opposite inequality, namely,
\[
\limsup\limits_{n\to\infty}\frac{d_n(\xi)}{2(d_1(\xi)+\cdots+d_{n-1}(\xi))}\leq \frac{\theta}{1-\theta}.
\]
Combining (\ref{6.10.2}) and (\ref{6.10.0}) establishes the first step of the proof.
%This works similarly starting from $d(0, z_{g_n})-n\kappa\leq 2(d_1(\xi)+\cdots+d_{n-1}(\xi))\leq d(0, z_{g_n})$.

\vspace{2mm}

\noindent We now aim to use this equivalent definition to find an upper bound for the sought-after Hausdorff dimension. So, suppose that $\xi\in\mathcal{J}_{\theta, \kappa}^*(G)$, that is, suppose that $\xi\in\mathcal{B}_\kappa$ is such that $\lim_{n\to\infty}d_n(\xi)=\infty$ and
\[
\limsup_{n\to\infty}\frac{d_n(\xi)}{2(d_1(\xi)+\cdots +d_{n-1}(\xi))}=\frac{\theta}{1-\theta}.
\]
Then, for each $n\in\N$ pick an integer $\hat{s}_n$ such that
\[
\log(\hat{s}_n)\leq d_n(\xi)<\log(\hat{s}_n +1).
\]
Since $\lim_{n\to\infty}d_n(\xi)=\infty$, we can immediately infer that $\lim_{n\to\infty}\hat{s}_n=\infty$. It is clear that we can write, say,
$\log \hat{s}_n\leq d_n(\xi)<\log 3\hat{s}_n$. So, if we can show that $\limsup\limits_{n\to\infty}\frac{\log \hat{s}_n}{2\log(\hat{s}_1\ldots \hat{s}_{n-1})}=\frac{\theta}{1-\theta}$, then we have that $\xi\in F_{\hat{s}, 3,\kappa}$. But,
\begin{eqnarray*}
\frac{d_n(\xi)}{2(d_1(\xi)+\cdots +d_{n-1}(\xi))}\leq \frac{\log(3\hat{s}_n)}{2\log(\hat{s}_1\ldots \hat{s}_{n-1})}
\end{eqnarray*}
and
\[\frac{d_n(\xi)}{2(d_1(\xi)+\cdots +d_{n-1}(\xi))}\geq \frac{\log(\hat{s}_n)}{2\log(\hat{s}_1\ldots \hat{s}_{n-1})+2n\log3}.
\]

From these two inequalities, we see that this reduces to basically the same argument again. Thus we obtain that
\[
\dim_H(\mathcal{J}_{\theta, \kappa}^*(G))\leq \dim_{H}(F_{\hat{s}, 3,\kappa})=\frac{1}{2(1+\frac{\theta}{1-\theta})}=\frac12 (1-\theta).
\]

Finally, suppose now that $N>2$ and $(s_n)_{n\in\N}$ is a sequence satisfying \[
\lim_{n\to\infty}s_n=\infty,\  \limsup_{n\to\infty}\frac{\log s_n}{\log(s_1\ldots s_{n-1})}=\frac{\theta}{1-\theta} \ \text{ and }\  \log s_n\leq d_n(\xi)<\log Ns_n\ \forall n\in\N.
\]Let $\xi\in F_{s, N,\kappa}(G)$.
By similar reasoning to that above, it follows that $\xi\in \mathcal{J}_{\theta, \kappa}^*(G)$. Consequently, we have that
\[
\frac12(1-\theta)=\dim_H(F_{s,N, \kappa}(G))\leq \dim_H(\mathcal{J}_{\theta, \kappa}^*(G)).
\]
Thus, combining this with the opposite inequality achieved above shows that $\dim_H(\mathcal{J}_{\theta, \kappa}^*(G))=1/2(1-\theta)$. As $\kappa>0$ was arbitrary, the proof of Theorem 2 is finished.

\end{proof}

\section{Weak Multifractal Spectra for the Patterson Measure}\label{final}

In this last section, we provide the proof of Theorem 3.

\begin{proof}[Proof of Theorem 3.]
     The global measure formula for $\mu$ gives the existence of
     a  constant
     $c>0$ (depending only on $G$), such that for each $\xi \in L(G)$ and
every $t>0$ we have that

     \[ \delta +(\delta -k(\xi_t)) \frac{\Delta(\xi_t)}{t} - \frac{c}{t}
     \leq \frac{\log \mu (b(\xi_t))}{\log e^{-t}} \leq \delta +(\delta -k(\xi_t))
      \frac{\Delta(\xi_t)}{t} + \frac{c}{t} .\]
(Here we are interested in the case that $t=t_n(\xi)$, $\Delta(\xi_t)=d_n(\xi)$ and $k(\xi_t)=1$.) From this we immediately deduce that
     $\xi \in {\mathcal J}_\theta^{*}(G)$ if and only if $\xi \in
     {\mathcal B}(G)$ and \[
     \limsup_{n \to \infty} \frac{\log \mu(b(\xi, e^{-t_n(\xi)}))}{-t_n(\xi)}
    =\delta -(1-\delta) \theta.\]
     Consequently, if $\beta:=  \delta -(1-\delta) \theta$, the result then follows immediately by an application of Theorem 2.
     \end{proof}

     \begin{rem} Note that in \cite{SJarnik1}
   a ``weak multifractal analysis'' of the Patterson measure was given. The
analysis there was based on
 investigations of the Hausdorff dimension  of the associated
{\em $\theta$-Jarn\'{\i}k limit set}
  \[ {\mathcal J}_\theta(G) : = \left\{ \xi \in L(G) : \limsup_{t \to
\infty} \frac{\Delta(\xi_t)}{t} \geq
\theta \right\}.
\]
 In
\cite{SJarnik1} (see also \cite{SJarnik2} and \cite{HV})
 the result was obtained that
\[ \dim_H({\mathcal J}_\theta(G)) = (1-\theta)\delta, \, \mbox{
for each } \, \theta \in [0,1].\]
In \cite{SJarnik2} it was then shown how to use this result in order
to derive the following  ``weak multifractal spectrum'' of the
Patterson measure:
\[
\dim_H\left({\mathcal F}_{\beta} \right) =  \left\{
\begin{array}
  {r@{\quad\text{for}\quad}l}0 & 0 <  \beta \leq 2\delta-1 \\
\delta\cdot f_p(\beta)
& 2\delta-1 < \beta \leq \delta \\
  \delta & \beta > \delta .
\end{array} \right.,
\]
where $f_p$ is given, as before, by $f_p(\beta)=(\beta-(2\delta-1))/(1-\delta)$ and where $\mathcal{F}_\beta(G)$ is defined by

\[
\mathcal{F}_\beta(G):=\left\{\xi\in L(G): \liminf_{n\to\infty}\frac{\log \mu(b(\xi, e^{-t_n(\xi)}))}{-t_n(\xi)}\leq \beta \right\}.
\]

The outcome here should be compared with the result in Theorem 3. The two spectra are illustrated in Figure 5.1, below.

\begin{figure}[htbp]\label{paddyspectra}
\begin{center}
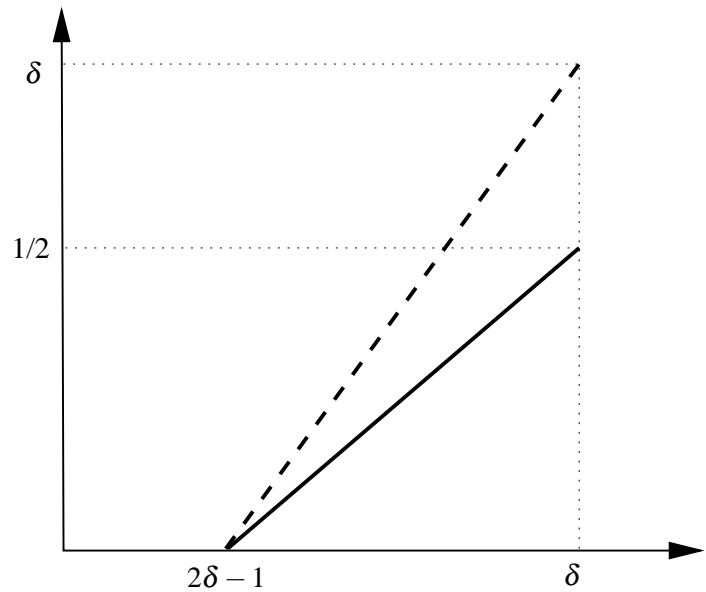\caption{Weak multifractal
spectra for the Patterson measure. The upper (dashed) line is the graph of the spectrum $\dim_H(\mathcal{F}_\beta(G))$ and the lower (solid) line is the graph of the spectrum $\dim_H(\mathcal{F}^*_\beta(G))$.}
\end{center}
\end{figure}

     \end{rem}

\thanks{The author would like to acknowledge the efforts of Prof. Urba\'{n}ski and the others on the organising committee for providing such a stimulating research environment during the Dynamical Systems II conference. The author would also like to thank her supervisor, Prof. Bernd Stratmann, for numerous helpful discussions.}

\end{document}